\documentclass[a4paper,11pt,reqno]{amsart} %
\usepackage[verbose]{geometry}
\usepackage[english]{babel}
\usepackage[utf8]{inputenc}

\usepackage{mathtools}
\usepackage{amsmath, amssymb, amsthm}
\usepackage{enumerate}
\usepackage{esint}              %

\usepackage{bbm}    %
\usepackage{mathrsfs}           %
\usepackage{cases}              %

\usepackage{graphicx}           %

\usepackage{tikz}
\usetikzlibrary{decorations, patterns}
\usepackage{ifthen}

\usepackage[hang, small, bf]{caption}

\DeclareFontFamily{U}{BOONDOX-calo}{\skewchar\font=45 }
\DeclareFontShape{U}{BOONDOX-calo}{m}{n}{
  <-> s*[1.05] BOONDOX-r-calo}{}
\DeclareFontShape{U}{BOONDOX-calo}{b}{n}{
  <-> s*[1.05] BOONDOX-b-calo}{}
\DeclareMathAlphabet{\mathcalboondox}{U}{BOONDOX-calo}{m}{n}
\SetMathAlphabet{\mathcalboondox}{bold}{U}{BOONDOX-calo}{b}{n}
\DeclareMathAlphabet{\mathbcalboondox}{U}{BOONDOX-calo}{b}{n}

\usepackage{hyperref}
\usepackage[noabbrev,capitalize]{cleveref} %
\crefformat{section}{\S#2#1#3} %
\crefformat{subsection}{\S#2#1#3}
\crefformat{subsubsection}{\S#2#1#3}

\usepackage{todonotes}
\newtheorem{theoremvar}{Theorem}
\newenvironment{theorem*}[1]
{\renewcommand{\thetheoremvar}{#1}\theoremvar}
{\endtheoremvar}

\newtheorem{corollaryvar}{Corollary}
\newenvironment{corollary*}[1]
{\renewcommand{\thecorollaryvar}{#1}\corollaryvar}
{\endtheoremvar}

\newtheorem{lemmavar}{Lemma}
\newenvironment{lemma*}[1]
{\renewcommand{\thelemmavar}{#1}\lemmavar}
{\endtheoremvar}

\theoremstyle{plain}
\newtheorem{thm}{Theorem} %
\newtheorem{prop}{Proposition}[section]      %

\newtheorem{lemma}[prop]{Lemma}
\newtheorem{teo}[thm]{Theorem}  %

\theoremstyle{plain} %

\theoremstyle{definition}

\newtheorem*{assumption*}{Assumption}

\newtheorem{define}[prop]{Definition}

\theoremstyle{remark}

\newtheorem{remark}[prop]{Remark}
\numberwithin{equation}{section} %

\newcommand{\1}{ \mathbbm{1}}   %
\newcommand{\D}{ \, \mathrm{d}} %
\renewcommand{\roman}[1]{%
  \textup{\uppercase\expandafter{\romannumeral#1}}%
}
\renewcommand{\roman}[1]{%
  \textup{\uppercase\expandafter{\romannumeral#1}}%
}

\DeclareFontFamily{U}{mathx}{\hyphenchar\font45}
\DeclareFontShape{U}{mathx}{m}{n}{
  <5> <6> <7> <8> <9> <10>
  <10.95> <12> <14.4> <17.28> <20.74> <24.88>
  mathx10
}{}
\DeclareSymbolFont{mathx}{U}{mathx}{m}{n}
\DeclareFontSubstitution{U}{mathx}{m}{n}
\DeclareMathAccent{\widecheck}{0}{mathx}{"71}
\DeclareMathAccent{\wideparen}{0}{mathx}{"75}

\renewcommand{\hat}[1]{\widehat{#1}} %

\DeclareMathAlphabet\EuScript{U}{eus}{m}{n}
\newcommand{\treestructure}{ \mathcal{T} } %
\newcommand{\tents}{ \mathcalboondox{T} } %
\providecommand{\testC}{ \mathfrak{T} } %

\hypersetup{
  pdftitle={Refined Two Weight Estimates for the Bergman projection},
  pdfauthor={Gianmarco Brocchi},
  colorlinks=true,
  citecolor=[rgb]{0.682, 0.133, 0.188},     %
  linkcolor=[rgb]{0,0,0.5},    %
  urlcolor=[rgb]{0,0,0.75},     
  pdfstartview=FitH,
  pdfdisplaydoctitle=true,
  pdflang=en-GB,
  breaklinks=true
}

\allowdisplaybreaks[4] %

\usepackage[backend=biber,backref=false,style=alphabetic]{biblatex}
\AtEveryBibitem{%
  \clearlist{language}%
}
\usepackage{csquotes}           %

\title{Refined two weight estimates for the Bergman projection}
\author[Brocchi]{Gianmarco Brocchi}
\address{
  School of Mathematics\\
  University of Birmingham\\
  Edgbaston\\
  Birmingham\\
  B15 2TT\\
  England}
\email{gianmarcobrocchi@gmail.com} %

\date{\today}

\begin{document}

\begin{abstract}
  We prove sufficient conditions for the two-weight boundedness
  of the Bergman projection on the unit ball.
  The first condition is in terms of Orlicz averages of the weights,
  while the second condition is in terms of the mixed $B_{\infty}$--$B_2$ characteristics.
\end{abstract}

\maketitle

\section{Introduction}
Let $\mathbb{B}^d$ be the complex unit ball in $\mathbb{C}^d$.
The Bergman projection from $L^2(\mathbb{B}^d)$ onto the subspace of holomorphic functions $A^2(\mathbb{B}^d)$
is the integral operator %
\begin{equation*}%
  Pf(z) \coloneqq \int_{\mathbb{B}^d} \frac{f(\zeta)}{(1 - z\bar{\zeta})^{d+1}} \D{\nu}(\zeta)
\end{equation*}
where $\D{\nu}$ is the normalised measure on $\mathbb{B}^d$.
It is a classical result \cite{ZJ64,zbMATH03466041} that the Bergman projection $P$
extends to a bounded operator on $L^p(\mathbb{B}^d)$ for all $p \in (1,\infty)$.
Weighted estimates were brought in this context by Békollè and Bonami \cite{BekolleBonami78, Bekolle82},
where the projection $P$ is seen as a singular integral operator
with respect to a particular pseudo-metric on the unit ball.
See also \cite{McNeal94} where this approach  has been extended to convex domains of finite type.
Békollè and Bonami proved that
given a real-valued, positive function $w \in L^1_{\mathrm{loc}}(\mathbb{B}^d)$,
the Bergman projection $P$ is bounded on the weighted space $L^p(\mathbb{B}^d, w \D{\nu})$
if and only if the weight $w$ satisfies the following condition
\begin{equation}\label{eq:classicalBp}\tag{$\mathcal{B}_p$}
  [w]_{\mathcal{B}_p} \coloneqq \sup_{B} \frac{1}{\lvert B \rvert} \int_{B} w \D{\nu} \left( \frac{1}{\lvert B \rvert} \int_{B} w^{1-p'} \D{\nu} \right)^{p-1} < \infty
\end{equation}
where the supremum is taken over all balls $B \subset \mathbb{B}^d$ with respect to the pseudo-metric mentioned above,
and $\lvert B \rvert \coloneqq \nu(B)$ denotes the measure of $B$.

The optimal dependence of the norm $\lVert P \rVert_{L^p(w) \to L^p(w)}$ on the Békollè--Bonami characteristic $[w]_{\mathcal{B}_p}$
has been obtained a few decades later: in dimension $d = 1$ by Pott and Reguera \cite{PR13},
and in higher dimension by Rahm, Tchoundja, and Wick \cite{RTW17}. They proved that
\begin{equation}\label{eq:L2-oneweight-bound}
  \lVert P \rVert_{L^2(w) \to L^2(w)} \le C \,  [w]_{\mathcal{B}_2}
\end{equation}  
where $C$ is a constant independent on the weight.
Estimate \eqref{eq:L2-oneweight-bound}
is analogous to the celebrated $A_2$-theorem \cite{A2}
from which optimal weighted $L^p$ estimates for $1 < p < \infty$ can be extrapolated,
by mean of the sharp extrapolation theorem \cite{MR2140200}:
\begin{equation*}
  \lVert P \rVert_{L^p(w) \to L^p(w)} \le C_p \,  [w]_{\mathcal{B}_p}^{\max\big\{\frac{1}{p-1}, 1 \big\}} .
\end{equation*}  

In the spirit of \cite{HP13involvingAinfty},
bound \eqref{eq:L2-oneweight-bound} has been refined in \cite{APR17}
by using a  dyadic version of the Békollè--Bonami class
and the $B_\infty$ characteristic introduced in \cite[\S 5]{APR17}.

We now define the joint Békollè--Bonami characteristic,
where the pseudo-balls  in \eqref{eq:classicalBp} are replaced with dyadic tents.
\begin{define}[Dyadic Békollè--Bonami weights]\label{def:dyadicB_p}
  For two weights $w,\sigma$ on $\mathbb{B}^d$ and $p\in (1,\infty)$
  their joint $B_p$ characteristic is
  \begin{equation}\label{eq:dyadicBp}
    [w,\sigma]_{B_p} \coloneqq \sup_{\hat{K} \in \tents} \frac{1}{\lvert \hat{K} \rvert} \int_{\hat{K}} w \D{\nu} \left(\frac{1}{\lvert \hat{K} \rvert} \int_{\hat{K}} \sigma \D{\nu} \right)^{p-1}
  \end{equation}
  where the supremum is taken over the collection of all dyadic tents $\tents$,
  whose construction is postponed to \cref{sec:dyadic_structure_ball}.
  The quantity $\lvert \hat{K} \rvert \coloneqq \nu(\hat{K})$ denotes the volume of the dyadic tent $\hat{K}$.
  When $\sigma = w^{1-p'}$ is the dual weight of $w$,
  the quantity $[w]_{B_p} \coloneqq [w,w^{1-p'}]_{B_p}$ is the $B_p$ characteristic of $w$.
  We  write  $w \in B_p$ if $[w]_{B_p}$ is finite.
\end{define}

The $B_\infty$ characteristic is defined as the Fujii--Wilson characteristic \cite{Wilson87,MR2359017}
for the Muckenhoupt $A_\infty$ class, by mean of the maximal operator
\begin{equation*}
  M f(z) \coloneqq \sup_{\widehat{K} \in \tents} \left(\frac{1}{\lvert \hat{K} \rvert} \int_{\hat{K}} \lvert f \rvert \D{\nu} \right) \1_{\widehat{K}}(z)
\end{equation*}
where $\1_{\hat{K}}$ is the indicator function on $\hat{K}$.

\begin{define}[$B_\infty$ class]\label{def:B_infty}
  A weight $\sigma$ belongs to the class $B_\infty$
  if the quantity
  \begin{equation*}
    [\sigma]_{B_\infty}  \coloneqq \sup_{\hat{K} \in \tents} \frac{1}{\sigma(\hat{K})} \int_{\hat{K}} M(\sigma \1_{\hat{K}}) \D{\nu}
  \end{equation*}
  is finite. We refer to $[\sigma]_{B_\infty} $ as the $B_\infty$ characteristic of the weight $\sigma$.
\end{define}

The class $B_\infty$ has been further studied in \cite{APR19}.
Weights in this class are in general not doubling,
see \cite[Counterexample 1 and Remark 5.1]{Duo+2016} and \cite{zbMATH07554393}.
Nevertheless, the subclass of $B_\infty$ studied in \cite{APR19} enjoys similar properties to the Muckenhoupt $A_\infty$ class.

In \cite[Theorem 5.7 and Corollary 5.9]{APR17} we have the following refinement of
estimate \eqref{eq:L2-oneweight-bound} in $d = 1$.
\begin{theorem*}{A}[Aleman, Pott and Reguera 2017]\label{teo:APR17}
  Let $w \in B_2$ be a weight on $\mathbb{B}^1$. The Bergman projection $P \colon L^2(\mathbb{B}^1,w) \to A^2(\mathbb{B}^1,w)$
  and the following estimate holds
  \begin{equation}\label{eq:Binfty-L2-oneweight-bound}
    \lVert P \rVert_{L^2(w) \to L^2(w)} \le C \, [w]_{B_2}^{1/2} \big( [w]_{B_\infty}^{1/2} + [w^{-1}]_{B_\infty}^{1/2} \big) .
  \end{equation}
\end{theorem*}

\vspace{.5em}

In this note we address the following question:
\begin{quote}
  What are the sufficient conditions on two weight $u$, $\omega$
  for the boundedness of the Bergman projection $P \,\colon L^2(u) \to A^2(\omega)$?
\end{quote}

Progress on this question has been obtained via sparse domination \cite{APR17,Sehba21}.
Sparse domination is a powerful technique in harmonic analysis
to control  operators by maximal averages.
As a consequence, since weighted estimates for the maximal operator are known,
these immediately translate into estimates for the dominated operator.
In the case of the unit disc $\mathbb{B}^1$, Aleman, Pott and Reguera \cite{APR17}
found sufficient and necessary conditions for weights
which are modulus of functions in the Bergman space $A^2(\mathbb{B}^1)$. %

This article presents new sufficient conditions for the two-weight boundedness of
  the Bergman projection on the unit ball.
Our first result generalises estimate \eqref{eq:Binfty-L2-oneweight-bound}
to the two-weight setting and to higher dimension.
\begin{teo}\label{teo:mixed-bounds-Bergman}
  Let $\sigma, w$ be two weights on $\mathbb{B}^d$
  in the class $B_\infty$ %
  such that their joint $B_2$ characteristic $[w, \sigma]_{B_2}$ is finite.
  The Bergman projection $P$ on $L^2(\mathbb{B}^d)$ satisfies the following bound
  \begin{equation*}
    \lVert P(\sigma \,\cdot) \rVert_{L^2(\sigma) \to L^2(w)} \le C \, [w, \sigma]_{B_2}^{1/2} \big( [w]_{B_\infty}^{1/2} + [\sigma]_{B_\infty}^{1/2} \big)
  \end{equation*}
  where $C$ is a positive constant independent of $\sigma$ and $w$.
\end{teo}

\begin{remark} A few remarks are in order.
  \begin{itemize}
  \item The classical $\mathcal{B}_2$ characteristic is defined using Carleson tents, which we recall in \cref{subsec:BBweights}.
    For two weights $w,\sigma$ in $B_\infty$ their classical  $\mathcal{B}_2$ characteristic
    and the dyadic one in \eqref{eq:dyadicBp} are comparable, see \cref{subsec:comparison}. 

  \item \cref{teo:mixed-bounds-Bergman} %
    extends \cref{teo:APR17} %
    to higher dimension
    and also to general weights $w, \sigma$ that are not dual to each other.

  \item For the one-weight theory,
    the estimate in \cref{teo:mixed-bounds-Bergman} improves on the $B_2$ estimates in \cite{RTW17},
    since $[\sigma]_{B_\infty} \le [\sigma]_{B_2}$.
    We recall the bound proved in the original paper \cite{APR17}
    in \cref{prop:Binfty_improvement}.
  \end{itemize}
\end{remark}

Our second result is a sufficient condition in terms of Orlicz averages of two weights.
Orlicz spaces generalise $L^p$ spaces and their norms are defined using Young functions, which we recall in \cref{subsec:Orlicz}.
\begin{define}\label{def:Orlicz_bump}
  Given two weights $w,\sigma$ and two Young functions $\Phi,\Psi$, %
  we define  the joint Orlicz bump as
  \begin{equation*}
    [w,\sigma]_{\Phi,\Psi} \coloneqq \sup_{\hat{K} \in \tents} \frac{\langle w\rangle_{\hat{K}}}{\langle w^{1/2} \rangle_{\Phi,\hat{K}}} \frac{\langle \sigma\rangle_{\hat{K}}}{\langle \sigma^{1/2} \rangle_{\Psi,\hat{K}}} 
  \end{equation*}
  where $\langle \,\cdot \,\rangle_{\Phi,\hat{K}}$ denotes the Orlicz average on the dyadic tent $\hat{K}$ with respect to $\Phi$.
  See \cref{subsec:Orlicz} for the precise definitions
  of Orlicz average and the associated maximal function.
\end{define}
We have the following result.
\begin{teo}\label{teo:Bumps_for_Bergman}             %
  Let $\sigma, w$ be two weights %
  on the unit ball $\mathbb{B}^d$ in $\mathbb{C}^d$ and let $\Phi, \Psi$ be two Young functions 
  such that the associated maximal function defined in \eqref{eq:maximal_function_Orlicz} is bounded on $L^2$.
  Then the Bergman projection $P$ on $L^2(\mathbb{B}^d)$ satisfies the following bound
  \begin{equation*}%
    \lVert P(\sigma \,\cdot) \rVert_{L^2(\sigma) \to L^2(w)} \le C \, [w,\sigma]_{\Phi,\Psi}
  \end{equation*}
  where $C$ is a positive constant  independent of $\sigma,w$.
\end{teo}
The condition in \cref{teo:Bumps_for_Bergman} is known as bump condition,
as the averages of the weights have been ``bumped up'' in the scale of Orlicz spaces.
\cref{teo:Bumps_for_Bergman} is deduced from a sparse operator dominating $P$.
In particular, %
it follows by combining the domination in \cite{RTW17} with the known estimates for sparse forms in \cite{Kangwei2weight}.
Nevertheless, to the best of our knowledge, these estimates
have not appeared in the context of Bergman spaces on the ball.
Previous results involving fractional Bergman operators on the upper-half plane can be found in \cite{Sehba18}.
Recently Sehba also deduced two-weight estimates for the Bergman projection on the upper-half plane $\mathbb{R}^2_+$
in terms of Sawyer testing conditions as the one in \cref{subsec:Sawyer_testing_sparse} for sparse operators,
see \cite[Theorem 2]{Sehba21}.

\begin{remark}
  Fundamental to our approach is the possibility to approximate
  Carleson tents with dyadic tents.
  As in \cite{RTW17}, we exploit the available dyadic structure on $\mathbb{B}^d$ developed by Arcozzi,  Rochberg, and Sawyer in \cite{ARS02}.
  We explain how this structure is constructed in \cref{sec:dyadic_structure_ball}.

  The construction of the dyadic family is independent of the two weights.
  If one could construct a family of subsets which is sparse
  with respect to the given weights, one could dispense of the $B_\infty$ condition in \cref{teo:mixed-bounds-Bergman}.
  Indeed, under this assumption the joint Berezin condition is a necessary and sufficient condition
  for the two-weight boundedness of $P$, see \cite{FangWang2015}.
  
    It is possible to construct
    a similar dyadic structure also on convex domains of finite type via the dyadic flow tents \cite{Dyadicflowtent20}.  
    This generalises the construction in \cref{sec:dyadic_structure_ball} for the ball.
    The resulting collection of dyadic flow tents is sparse, %
    and it produces %
    weighted estimates for the Bergman projection on convex domains of finite type.
    Since our bump condition implies the boundedness of a sparse operator,
    the same condition implies the boundedness of the Bergman projection on convex domains of finite type
    by the pointwise control in \cite[Lemma 4.1]{Dyadicflowtent20}.
    Similar results on classes of pseudoconvex domains can also be found in \cite{MR4263006}.
\end{remark}

\section{Preliminaries}
We recall a few classical definitions and we introduce some notations.
For two positive quantities $X$ and $Y$,
we write $X \lesssim Y$ if there exists a constant $C > 0$ such that $X \le C \, Y$.
In particular, this constant is independent of the weights.
We write $X \eqsim Y$ is both $X \lesssim Y$ and $Y \lesssim X$ holds, possibly with different constants.

Given a Borel set $E \subset \mathbb{B}^d$ and a locally integrable function $f$,
we denote the average of $f$ on $E$ with respect to the measure $\D{\nu}$ by
\begin{equation*}
  \langle f \rangle_{E} \coloneqq \frac{1}{\nu( E )} \int_E f \D{\nu} .
\end{equation*}

We remind the reader that, although some of the results below
are for general $L^p$ spaces, only $L^2$ results are needed for our argument.

\subsection{Classical Békollè--Bonami weights}\label{subsec:BBweights}
We recall the definition of a Békollè--Bonami weight on the unit ball.
These weights satisfy an $A_p$ condition
where the role of cubes is played by Carleson tents.

\begin{define}[Carleson tent on the unit ball]
Given a point $z \in \mathbb{B}^d \setminus \{ 0 \}$, consider the following set %
\begin{equation*}\textstyle
  T_z := \left\{ \zeta \in \mathbb{B}^d \; \colon \; \lvert 1 - \langle \zeta , \frac{z}{\lvert z\rvert} \rangle\rvert \le 1 - \lvert z\rvert  \right\} .
\end{equation*}
When $d = 1$ the set $T_z$ is the intersection of $\mathbb{B}^1$ with the disc centred at $z/\lvert z\rvert$ with radius $1 - \lvert z \rvert$,
whose boundary contains the point $z$.
When $z = 0$, we set $T_0 = \mathbb{B}^d$. 
\end{define}

\begin{define}[Békollè--Bonami weights]
  Given $p \in (1,\infty)$ and two weights $w,\sigma$ on $\mathbb{B}^d$,
  we define their joint $\mathcal{B}_p$ characteristic:
  \begin{equation*}             %
    [w,\sigma]_{\mathcal{B}_p} \coloneqq \sup_{z \in \mathbb{B}^d} \langle w \rangle_{T_z} \langle \sigma \rangle_{T_z}^{p-1} 
  \end{equation*}
  where $\langle w \rangle_{T_z} \coloneqq \lvert T_z \rvert^{-1} \int_{T_z} w $, and $\lvert T_z \rvert$ denotes the volume of the tent $T_z$.
  As before, we denote by $[w]_{\mathcal{B}_p} \coloneqq [w,w^{1-p'}]_{\mathcal{B}_p}$.
  We say that $w \in \mathcal{B}_p$ if $[w]_{\mathcal{B}_p}$ is finite.
\end{define}
\begin{remark}
  This classical definition  is quantitatively equivalent to the dyadic one in \cref{def:dyadicB_p}, see \cref{subsec:comparison}.
\end{remark}

\subsection{Orlicz average}\label{subsec:Orlicz}
We recall the definition  of Orlicz averages used in the statement of \cref{teo:Bumps_for_Bergman}.
We first need to introduce Young functions.
\begin{define}[Young function]
  Let $\Phi \colon [0,\infty) \to [0,\infty)$ be a continuous, convex,
  strictly increasing function such that
  \begin{equation*}
    \Phi(0) = 0 \quad \text{ and } \quad \lim_{t \to +\infty}\frac{\Phi(t)}t = + \infty .
  \end{equation*}
\end{define}

Given a set $Q$ and a Young function $\Phi$ we define
the Orlicz average via the Luxembourg norm
\begin{equation*}
  \langle f \rangle_{\Phi,Q} \coloneqq \inf \{ \lambda > 0 \, : \, \langle \Phi(f/\lambda) \rangle_Q \le 1 \}.
\end{equation*}

For example, when $\Phi(t) = t^p$ with $1 < p < \infty$
the Orlicz average corresponds to the usual $L^p$ average on the set $Q$.
Given a Young function $\Phi$ and a collection of sets,
one can consider the maximal function
\begin{equation}\label{eq:maximal_function_Orlicz}
  M_\Phi f \coloneqq \sup_{Q} \langle \lvert f\rvert \rangle_{\Phi,Q} \1_Q 
\end{equation}
where the supremum is taken over all sets in the collection.
In the context of this paper, the sets $Q$ are the dyadic tents $\hat{K}$
whose construction can be found in \cref{sec:dyadic_structure_ball}.

In \cite[Theorem 1.7]{Perez1995} Pérez characterised the Young functions for which 
the associated maximal function is bounded on $L^p$.
\begin{theorem*}{B}[Pérez 1995]\label{teo:Perez95}
  Given a Young function $\Phi$,
  the associated maximal function $M_{\Phi}$ maps $L^p$ to $L^p$, for $1<p<\infty$, if and only if
  \begin{equation}
    \label{eq:Bp_condition}\tag{$\mathscr{B}_p$}
    \int_1^{\infty} \frac{\Phi(t)}{t^p} \frac{\D{t}}t < + \infty .
  \end{equation}
\end{theorem*}

Note that the operator $M_\Phi$ is also bounded on $L^\infty$ \cite[Lemma 3.2]{Anderson2015}. %
We say that a Young function $\Phi$ belongs to $\mathscr{B}_p$ if the condition \eqref{eq:Bp_condition} holds.
We shall not confuse the class of weights $B_p$ and $\mathcal{B}_p$ with the one of Young functions in $\mathscr{B}_p$.
To help the reader, we will always denote Young functions by capital Greek letters ($\Phi$ or $\Psi$),
whilst we keep the lower case notation $w,v,u, \sigma$ for weights.

The proof of \cref{teo:mixed-bounds-Bergman} exploits testing conditions for sparse operators,
which we now introduce.
\subsection{Sparse families and sparse operators}
We start by recalling the definition of a sparse family.
\begin{define}[Sparse collection]\label{def:sparse}
  A collection of sets $\mathscr{S}$ is $\frac{1}{\tau}$-sparse, for $\tau \ge 1$, if for
  any $Q \in \mathscr{S}$ there exists a subset $E_Q \subseteq Q$ such
  that $\{ E_Q \}_{Q\in\mathscr{S}}$ are pairwise disjoint and $\lvert Q \rvert \le \tau \lvert E_Q \rvert$.
\end{define}

\begin{remark}
  The definition above uses the Lebesgue measure,
  and it can be applied to other measures as well,
  although these notions are not,
  in general, equivalent.

  Nevertheless, if $\mathscr{S}$ is a sparse collection with respect to the measure $\D{\nu}$,
  and $\sigma$ is a $B_p$ weight,
  then $\mathscr{S}$ is also sparse with respect to the measure $\sigma \D{\nu}$.
  See \cref{rmk:equivalent-sparse}.
\end{remark}
Let $\mathscr{S}$ be a sparse collection.
We denote by $\Lambda_{\mathscr{S}}$ the corresponding sparse operator
\begin{equation}\label{eq:sparse_operator}
  \Lambda_{\mathscr{S}}f \coloneqq \sum_{Q \in \mathscr{S}} \langle f \rangle_{Q} \1_Q .
\end{equation}

\subsection{Sawyer testing conditions}\label{subsec:Sawyer_testing_sparse}

The weights for which $\Lambda_{\mathscr{S}} (\sigma \,\cdot)\colon L^p(\sigma) \to L^p(w)$ holds,
as well as the equivalent dual formulation $\Lambda_{\mathscr{S}} (\cdot \, w)\colon L^{p'}(w) \to L^{p'}(\sigma)$,
have been characterised by Sawyer in terms of following testing conditions: 

\begin{equation}
  \begin{aligned}
    \lVert \Lambda_{\mathscr{S}}(\sigma \1_Q) \rVert_{L^p(w)}^p & \le \testC \sigma(Q) \, , \quad \forall\, Q \in \mathscr{S} \\
    \lVert \Lambda_{\mathscr{S}}(w \1_Q) \rVert_{L^{p'}(\sigma)}^{p'} & \le \testC' w(Q)  \, , \quad \forall\, Q \in \mathscr{S} %
  \end{aligned}\label{eq:testing_condition_on_sparse}
\end{equation}
where the optimal testing constants are the finite quantities %
\begin{equation}
  \begin{aligned}
    \testC & \coloneqq \testC_p(w,\sigma) \coloneqq \sup_Q \frac{\lVert \1_Q \Lambda_{\mathscr{S}}(\1_Q \sigma) \rVert_{L^p(w)}^p}{\sigma(Q)} \, , \\
    \testC' & \coloneqq \testC'_{p'}(w,\sigma) \coloneqq \sup_Q \frac{\lVert \1_Q \Lambda_{\mathscr{S}}(\1_Q w) \rVert_{L^{p'}(\sigma)}^{p'}}{w(Q)} .
  \end{aligned}\label{eq:testing_constants}
\end{equation}

These conditions are named after Sawyer,
who first derived them for maximal operators \cite{Sawyer82}
and for fractional and Poisson integrals \cite{Sawyer88}.
For sparse operators they have been proved in \cite{LSUT09}.

Testing constants for off-diagonal estimates
$\Lambda_{\mathscr{S}} (\sigma \,\cdot)\colon L^p(\sigma) \to L^q(w)$ for $q \neq p$
and more general sparse forms have also been studied, see \cite[Theorem 1.1]{Kangwei2weight}.
In particular we have
\begin{equation*}
  \lVert \Lambda_{\mathscr{S}}(\sigma \, \cdot)\rVert_{L^p(\sigma) \to L^p(w)} \eqsim \big( \testC^{1/p} + (\testC')^{1/p'} \big) .
\end{equation*}

In the proof of \cref{teo:mixed-bounds-Bergman} we estimate the constants $\testC, \testC'$ from above
with the Békollè--Bonami characteristic of the weights $w,\sigma$.

\subsection{Program to deduce weighted estimates}\label{subsec:plan}
A possible route to prove weighted estimates for $P$ follows these steps, see also \cite{Lerner2013}.
\begin{enumerate}
\item \textit{(Control by a positive operator).} The modulus of the Bergman projection is controlled by 
  the maximal Bergman operator:
  \begin{equation*}
    P^+f(z) \coloneqq \int_{\mathbb{B}^d} \frac{f(\zeta)}{\lvert 1 - z\bar{\zeta}\rvert^{d+1}} \D{\nu}(\zeta).
  \end{equation*}
  Namely we have $\lvert Pf(z) \rvert \le P^+\lvert f\rvert(z)$. 
  Note that $P^+(\lvert\, \cdot\, \rvert)$ is a real-valued, positive operator. 
  For positive weights $w,v$ we have
  \begin{equation*}
    \lVert P \rVert_{L^2(w) \to L^2(v)} \le \lVert P^+(\lvert\, \cdot\, \rvert) \rVert_{L^2(w) \to L^2(v)}.
  \end{equation*}

\item \textit{(Equivalence with a sparse operator).}
  Once the dyadic structure on $\mathbb{B}^d$ is constructed (see \cref{sec:dyadic_structure_ball}), 
  the collection of dyadic tents $\mathcalboondox{T}$ is sparse, see \cref{lemma:sparsity_of_dyadic_tents}.
  The associated sparse operator $\Lambda_{\tents}$ is equivalent to the
  maximal Bergman operator:
  \begin{equation*}
    P^+\lvert f\rvert(z) \eqsim_d \Lambda_{\tents}\lvert f \rvert(z)  \coloneqq \sum_{\hat{K}_\alpha \in \tents} \langle \lvert f \rvert \rangle_{\hat{K}_\alpha} \1_{\hat{K}_\alpha}.
  \end{equation*}
  See \cref{lemma:control_P_by_sparse}, and \cite[Lemma 5]{RahmWick17} for a proof.
  
\item \textit{(Bumps for the sparse operator).} Two-weight estimates for sparse operators are well understood.
  For example, they are equivalent to two-weight estimates for the maximal operator $M$, see \cite[Theorem 1.2]{Lerner2013}.
  Sufficient conditions on $(w,v)$ for the boundedness of
  \begin{equation*}
    \lVert M \rVert_{ L^2(w) \to L^2(v) } \quad \text{ and } \quad \lVert \Lambda_{\tents} \rVert_{ L^2(w) \to L^2(v) } 
  \end{equation*}
  are known in terms of the testing conditions presented in \cref{subsec:Sawyer_testing_sparse}. %
\end{enumerate}

The task of characterising weights $w,v$
for which $\lVert P \rVert_{L^2(w) \to L^2(v)}$ is finite is still open.

\subsection{Dyadic structure on the complex unit ball}\label{sec:dyadic_structure_ball}
We borrow the dyadic structure on the ball developed by
Arcozzi,  Rochberg, and Sawyer \cite[\S 2.2]{ARS02} and also used in \cite[\S 2]{RTW17}.
This structure introduces a collection of sets called ``dyadic kubes'',
which comes with a tree structure $\treestructure$ called Bergman tree
(namely a collection of partially ordered indexes $\{ \alpha \in \treestructure \}$.
The points $\{c_\alpha\}_{\alpha \in \treestructure}$ are the centres of the dyadic kubes).

We explain how the dyadic structure is constructed.

Let $\varphi_z$ be the bi-holomorphic involution of the ball exchanging $z$ and the origin:
\begin{equation*}%
  \varphi_z(w) \coloneqq \frac{z - \langle w, \frac{z}{\lvert z\rvert } \rangle\frac{z}{\lvert z\rvert } - \sqrt{1 -\lvert z \rvert^2} (w - \langle w, \frac{z}{\lvert z\rvert } \rangle\frac{z}{\lvert z\rvert } )}{1 - \langle w, z\rangle}.
\end{equation*}
The Bergman metric on the unit ball $\mathbb{B}^d$ is defined as
\begin{equation*}
  \beta(z,w) \coloneqq \frac12 \log \frac{1 + \lvert \varphi_z(w) \rvert}{1 - \lvert \varphi_z(w) \rvert}.
\end{equation*}
In the following, $B(z_0,r) \subset \mathbb{B}^d$ denotes the ball of centre $z_0$ and radius $r$ in the Bergman metric.
We also denote by $\mathbb{S}_{r}$ the sphere of radius $r$ centred at the origin, so $\mathbb{S}_r = \partial B(0,r)$.

Fix $R, \delta > 0$.
For $n \in \mathbb{N}$, there is a collection of points $\{ z_j^n \}_{j =1}^{J_n}$
and a partition of the sphere $\mathbb{S}_{n R}$ in Borel subsets $\{\Omega_j^n\}_{j=1}^{J_n}$ 
such that 
\begin{enumerate}
\item[(i)] $ \mathbb{S}_{n R} = \bigsqcup_{j = 1}^{J_n} \Omega_j^n $ ;
\item[(ii)] $ \big( B(z_j,\delta) \cap \mathbb{S}_{n R} \big) \subseteq \Omega_j^n \subseteq   \big( B(z_j, C \delta) \cap \mathbb{S}_{n R} \big)$ for some $C>0$.
\end{enumerate}

Let $\pi_{n R}$ denote the radial projection from $\mathbb{B}^d$ onto the sphere $\mathbb{S}_{n R}$.
The kubes are given by
\begin{align*}
  K_1^0 & \coloneqq B(0,R) , \\
  K_j^n & \coloneqq \{ \zeta \in B(0,(n+1)R) \setminus B(0,n R) \, \colon\, \pi_{n R}(\zeta) \in \Omega_j^n \}. 
\end{align*}
The centre of the kube $K_j^n$ is $c_j^n \coloneqq \pi_{(n+\frac12)R}(z_j^n)$.
We say that a point $c_i^{n+1}$ is a child of $c_k^n$ if $\pi_{n R}(c_i^{n+1}) \in \Omega_k^n$.
Then the centres form a tree structure $\treestructure$,
which we will refer to as Bergman tree.

\begin{figure}[h]
  \centering
  \def\svgwidth{.6\textwidth}
  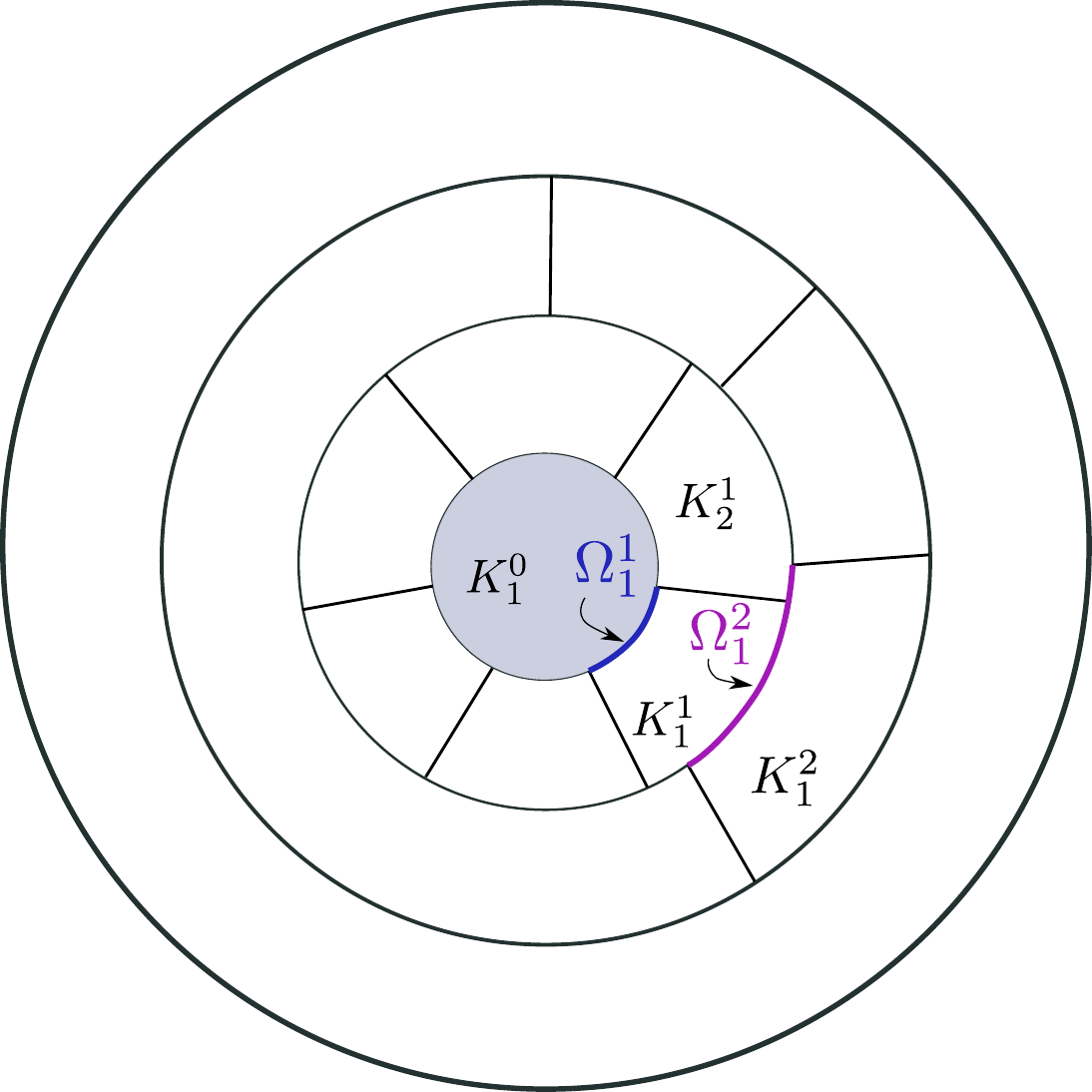  
  \caption{An example of the first generations of kubes and the respective $\Omega^n_k$ in the dyadic structure on $\mathbb{B}^1$.
  The structure for $\mathbb{B}^d$ is of a similar spirit.} \label{fig:kubes}
\end{figure}

To simplify the notation, 
let $\alpha$  be an element in $\treestructure$.
We denote by $K_{\alpha}$ the unique kube with centre $\alpha$.
If $\beta$ is a descendant of $\alpha$ we write $\beta \ge \alpha$.
Given a kube $K_\alpha$, the dyadic tent $\hat{K}_\alpha$ is the union of all kubes whose centres are descendant of $\alpha$ in $\treestructure$, namely
\begin{equation*}
  \hat{K}_\alpha \coloneqq \bigcup_{\beta \ge \alpha} K_\beta .
\end{equation*}

The volume of $K_\alpha$ and $\widehat{K}_{\alpha}$ are comparable.
This was originally proved in \cite[Lemma 2.8]{ARS06}, see also \cite[Lemma 1]{RTW17}.
\begin{lemma}[Arcozzi, Rochberg, and Sawyer 2006]\label{lemma:sparsity_of_dyadic_tents}
  Let $\treestructure$ be a Bergman tree on $\mathbb{B}^d$ with parameters $R,\delta$.
  There is a universal constant $\tau > 1$, depending only on $R, \delta$ and the dimension $d$,
  such that $\lvert \widehat{K}_\alpha \rvert \le \tau\lvert K_\alpha \rvert$ for all $\alpha \in \treestructure$.
\end{lemma}

From this lemma,
and from the fact that the kubes $\{ K_\alpha \}_{\alpha \in \treestructure}$ are pairwise disjoint,
it follows immediately that the collection of dyadic tents $\tents \coloneqq \{ \hat{K}_{\alpha}\}_{\alpha \in \treestructure}$
is $\frac{1}{\tau}$-sparse, in the sense of \cref{def:sparse}.

\begin{remark}\label{rmk:equivalent-sparse}
  Note that if $\sigma \in B_p$, by Hölder's inequality
  the collection $\tents$ is $(\tau^p [\sigma]_{B_p})^{-1}$-sparse with respect to the measure $\sigma \D{\nu}$.
\end{remark}

The characteristic $[\sigma]_{B_\infty}$ is controlled by $[\sigma]_{B_p}$.
We recall the proof from \cite[Proposition 5.6]{APR17}.

\begin{prop}[Aleman, Pott, Reguera 2017]\label{prop:Binfty_improvement}
  For $1< p < \infty$,
  let $w$ be a weight in $B_p$.
  Then we have
  \begin{equation*}
    [w]_{B_\infty} \le [w]_{B_p} .
  \end{equation*}
\end{prop}
\begin{proof}
  Let $w \in B_p$ and let $\sigma \coloneqq w^{1-p'}$ be the dual weight.
  By writing $1 = \sigma^{\frac1{p'}} \sigma^{\frac1{p} -1}$ and using Hölder's inequality, we have
  \begin{align*}
    \int_{\hat{K}} M(w \1_{\hat{K}} ) \sigma^{\frac1{p'}} \sigma^{\frac1{p} -1} & \le \left( \int_{\hat{K}} M(w \1_{\hat{K}})^{p'} \sigma \right)^{1/p'} \left( \int_{\hat{K}} \sigma^{1-p} \right)^{1/p} \\
                                                                            & \le \lVert M \rVert_{L^{p'}(\sigma) \to L^{p'}(\sigma)} \left( \int_{\hat{K}} w^{p'} \sigma \right)^{1/p'} \left( \int_{\hat{K}} \sigma^{1-p} \right)^{1/p} \\
                                                                            & \lesssim_{p,d} [w]_{B_p} \int_{\hat{K}} w \\    
  \end{align*}
  where we used that $w^{p'} \sigma =  w =  \sigma^{1-p} $
  together with Buckley's estimate \cite[Theorem 2.5]{BuckleyEstimates}
  for the Hardy--Littlewood maximal function:  %
  \begin{equation*}
    \lVert M \rVert_{L^{p'}(\sigma) \to L^{p'}(\sigma)} \lesssim_{p',d} [\sigma]_{B_{p'}}^{1/(p'-1)} = [w]_{B_p} .
  \end{equation*}
  A simple proof of the bound for the norm of $M$ can also be found in \cite{MR2399047}.
\end{proof}

\section{Proof of Theorem \ref{teo:mixed-bounds-Bergman}}
We start by noting that
the maximal Bergman operator $P^+$ is controlled by the sparse operator $\Lambda_{\tents}$ defined in \eqref{eq:sparse_operator}.
\begin{lemma}[{\cite[Lemma 5]{RTW17}}]\label{lemma:control_P_by_sparse}
  There exists a finite collection of Bergman trees  $\{ \treestructure_\ell \}_{\ell = 1}^N$
  such that
  \begin{equation*}
    P^+\lvert f\rvert(z) \eqsim \Lambda_{\tents} \lvert f\rvert(z) = \sum_{\hat{K}_\alpha \in \tents} \langle \lvert f \rvert \rangle_{\hat{K}_\alpha} \1_{\hat{K}_\alpha}
  \end{equation*}
  where $\tents \coloneqq \cup_{\ell = 1}^N \{ \hat{K}_\alpha \, \colon \, \alpha \in \treestructure_\ell \}$ is a sparse collection of dyadic tents.
\end{lemma}

Then \cref{teo:mixed-bounds-Bergman} and \cref{teo:Bumps_for_Bergman} follow from the respective estimates for $\Lambda_{\tents}$.
In the rest of the paper we prove of these estimates for a sparse operator $\Lambda_{\mathscr{S}}$ associated to a generic sparse collection $\mathscr{S}$.

The testing conditions for the boundedness of $  \lVert \Lambda_{\mathscr{S}}(\sigma \, \cdot)\rVert_{L^p(\sigma) \to L^p(w)} $  are
\begin{align*}
  \lVert \1_{\hat{K}_0} \Lambda_{\tents}( \sigma \1_{\hat{K}_0} )  \rVert_{L^2(w)}^2 & \lesssim [w,\sigma]_{B_2} [\sigma]_{B_\infty} \sigma(\hat{K}_0) , \\
  \lVert \1_{\hat{K}_0} \Lambda_{\tents}( w \1_{\hat{K}_0} ) \rVert_{L^2(\sigma)}^2 & \lesssim [\sigma,w]_{B_2} [w]_{B_\infty} w(\hat{K}_0) .
\end{align*}
By symmetry, it is enough to prove one of the two inequalities.
We choose the first one.
\begin{prop}\label{prop:B2-Binfty-bound}
  Let $\sigma, w$ be two weights. Then
  for any dyadic tent $\hat{K}_0 \in \tents$,
  we have
  \begin{equation*}%
    \lVert \1_{\hat{K}_0} \Lambda_{\tents} \sigma \rVert_{L^2(w)}^2 \lesssim [w,\sigma]_{B_2} [\sigma]_{B_\infty} \sigma(\hat{K}_0) .
  \end{equation*}
\end{prop}

We refer the reader to \cite[Prop. 5.2]{HytLac12} for a version of this result for dyadic shifts.
Since we deal with sparse operators, the proof we present here is simpler.
It follows the approach in Hytönen's work \cite[\S 5.A]{A_2_remarks} and in \cite[\S 5]{APR17}.

\begin{proof}[Proof of \cref{prop:B2-Binfty-bound}]

  For simplicity, we denote by $L_0 \in \tents$ a fixed dyadic tent, instead of $\hat{K}_0$.
  Recall that, since $\tents$ is sparse,
  there is a fixed $\tau \ge 1$ such that
  for every $L \in \tents$ there exists a subset $E_L \subseteq L$ with the property that $\lvert L \rvert \le \tau \lvert E_L\rvert$
  and the sets in $\{E_L \, :\, L \in \tents\}$ are pairwise disjoint.
  Then we have
  \begin{align*}
    \lVert  \1_{L_0} \Lambda_{\tents} \sigma \rVert_{L^2(w)}^2 & = \int_{L_0} \Big( \sum_{L \in \tents} \langle \sigma \rangle_L \1_L \Big)^2 w \\
                                                             & \le 2 \int_{L_0} \sum_{\substack{L \in \tents \\ L \subseteq L_0}} \langle \sigma \rangle_L \sum_{\substack{L' \in \tents \\ L' \subseteq L}} \langle \sigma \rangle_{L'} \1_{L'} \, w \\
    & = 2 \sum_{\substack{L \in \tents \\ L \subseteq L_0}} \langle \sigma \rangle_L \sum_{\substack{L' \in \tents \\ L' \subseteq L}} \langle \sigma \rangle_{L'} \langle w \rangle_{L'} \lvert L' \rvert \\
    & \le 2\sup_{L' \in \tents} \langle \sigma \rangle_{L'} \langle w \rangle_{L'}  \sum_{\substack{L \in \tents \\ L \subseteq L_0}} \langle \sigma \rangle_L \sum_{\substack{L' \in \tents \\ L' \subseteq L}} \lvert L' \rvert \\
    & \le 2 \tau \sup_{L' \in \tents} \langle \sigma \rangle_{L'} \langle w \rangle_{L'}  \sum_{\substack{L \in \tents \\ L \subseteq L_0}} \langle \sigma \rangle_L \lvert L \rvert \\
     & \lesssim [ \sigma, w ]_{B_2} \sum_{\substack{L \in \tents \\ L \subseteq L_0}} \langle \sigma \rangle_L \lvert L \rvert .
  \end{align*}
  The remaining sum is controlled by using the maximal function and the sparseness property. %
  We have
  \begin{align*}
    [ \sigma, w ]_{B_2} \sum_{\substack{L \in \tents \\ L \subseteq L_0}} \langle \sigma \rangle_L \lvert L \rvert & \le [ \sigma, w ]_{B_2} \sum_{\substack{L \in \tents \\ L \subseteq L_0}} \inf_L M( \sigma \1_{L_0} ) \lvert L \rvert  \\
                                                         & \le \tau [ \sigma, w ]_{B_2} \sum_{\substack{L \in \tents \\ L \subseteq L_0}} \int_{E_L} M( \sigma \1_{L_0} )  \\
    & \le \tau [ \sigma, w ]_{B_2} \frac{1}{\sigma(L_0)} \int_{L_0} M( \sigma \1_{L_0} ) \, \sigma(L_0) \\
    & \le \tau [ \sigma, w ]_{B_2} \Big( \sup_{L_0 \in \tents} \frac{1}{\sigma(L_0)} \int_{L_0} M( \sigma \1_{L_0} ) \Big)\, \sigma(L_0) \\
    & \lesssim [ \sigma, w ]_{B_2} [\sigma]_{B_\infty} \sigma(L_0) .
  \end{align*}

  This concludes the proof of the proposition.
\end{proof}

The proof of \cref{teo:mixed-bounds-Bergman} follows by combining the sparse domination in \cref{lemma:control_P_by_sparse}
with the bound for the sparse operator in \cref{prop:B2-Binfty-bound}.
This gives the bound
  \begin{equation*}
    \lVert P(\sigma \,\cdot) \rVert_{L^2(\sigma) \to L^2(w)} \le C \, [w, \sigma]_{B_2}^{1/2} \big( [w]_{B_\infty}^{1/2} + [\sigma]_{B_\infty}^{1/2} \big).
  \end{equation*}

  \subsection{Comparison of dyadic and classical characteristics}\label{subsec:comparison}

  We conclude by comparing the volume of a Carleson tent $T_z$  
with the volume of a dyadic tent $\hat{K}_\alpha$.
This is the content of the following two lemmas,
which use the concept of Bergman tree introduced in \cref{sec:dyadic_structure_ball};
see \cite[Lemma 3]{RTW17} and \cite[Lemma 2.4]{MR4134894}.
\begin{lemma}[Rahm, Tchoundja, and Wick, 2017]\label{lemma:volume_comparison}
  There exists a finite collection of Bergman trees $\{ \treestructure_\ell \}_{\ell = 1}^N$ such that
  for any tent $T_z$ there is $\ell\in\{1,\dots,N\}$ and $\alpha$ in $\treestructure_\ell$ such that
  $\hat{K}_{\alpha} \supseteq T_z$ and $\lvert T_z \rvert \eqsim \lvert \hat{K}_\alpha\rvert$.
\end{lemma}

Note that since a finite union of sparse families is sparse,
if we denote by
\begin{equation*}
  \tents \coloneqq \bigcup_{\ell = 1}^N \tents_\ell \quad \text{ where } \quad \tents_\ell \coloneqq \{ \hat{K}_\alpha \, \colon \, \alpha \in \treestructure_\ell \} \,,
\end{equation*}
then $\tents$ is a sparse collection of sets in the unit ball $\mathbb{B}^d$.

\begin{lemma}[Huo and Wick 2020]
  For any dyadic tent $\hat{K}_{\beta} \in \tents$ there exists a Carleson tent $T_z$
  such that $\hat{K}_{\beta} \subseteq T_z$ and $\lvert \hat{K}_{\beta} \rvert \eqsim \lvert T_z \rvert$. 
\end{lemma}

This result is proved for the disc \cite[Lemma 2.4]{MR4134894};
 the argument can be adapted for $d \ge 2$.
Then it holds that $[w,\sigma]_{B_2} \eqsim [w,\sigma]_{\mathcal{B}_2}$. %
The proof of \cref{teo:mixed-bounds-Bergman} is concluded. \qed

\section{Proof of Theorem \ref{teo:Bumps_for_Bergman}}
We derive a bump condition in $L^2$ for two weights $w,\sigma$ in terms of Orlicz averages.

We follow the approach in \cite[Theorem 5.2]{Kangwei2weight} and \cite[Theorem 6.1]{A_2_remarks}.

\begin{prop}\label{prop:bump_for_sparse}
  Let $\Lambda_{\mathscr{S}}$ be the sparse operator defined in \eqref{eq:sparse_operator}.
  For two weights $w,\sigma$ and two Young functions $\Phi, \Psi \in \mathscr{B}_2$,
  it holds
  \begin{equation*}
    \lVert \Lambda_{\mathscr{S}}( \sigma \,\cdot) \rVert_{L^2(\sigma) \to L^2(w)} \lesssim [\sigma,w]_{\Phi,\Psi} .
  \end{equation*}
\end{prop}

We split the proof of \cref{prop:bump_for_sparse} in a few simple steps.
We will use the notation $\langle f \rangle_Q^{\sigma} \coloneqq \sigma(Q)^{-1} \int_Q f \sigma$
and the following lemmata for $p=2$.
\begin{lemma}\label{lemma:aux1}
  Let $\sigma$ be a weight and let $\mathscr{S}$ be a $\frac{1}{\tau}$-sparse family with respect to the measure $\sigma \D{\nu}$.
  For $1< p < \infty$ and a function $f$ we have
  \begin{equation*}
    \Big( \sum_{F \in \mathscr{S}} (\langle f\rangle_F^{\sigma} )^p \sigma(F) \Big)^{1/p} \lesssim_{\tau,p} \lVert f\rVert_{L^p(\sigma)}
  \end{equation*}
  where the implicit constant depends only on the sparse family and on the exponent $p$.
\end{lemma}
\begin{proof}
  Since $\mathscr{S}$ is $\frac1{\tau}$-sparse, %
  for every $F \in \mathscr{S}$ there is $E_F \subseteq F$ with $\sigma(F) \le \tau \sigma(E_F)$,
  and the $\{E_F\,:\, F \in \mathscr{S}\}$ are disjoint.
  Let $M^\sigma$ be the maximal function defined by
  \begin{equation*}
    M^\sigma f \coloneqq \sup_{F \in \mathscr{S}} \langle \lvert f\rvert\rangle_F^{\sigma} \1_F .
  \end{equation*}
  We bound
  \begin{align*}
    \sum_{F \in \mathscr{S}} (\langle f\rangle_F^\sigma )^p \sigma(F) & \le \tau \sum_{F \in \mathscr{S}} \big(\inf_{E_F} M^{\sigma}f\big)^p \sigma(E_F) \\
                                                                  & \le \tau \sum_F \int_{E_F} \lvert M^{\sigma}f\rvert^p \sigma \D{\nu} \\
                                                                  & \le \tau \lVert M^{\sigma} \rVert_{L^p(\sigma) \to L^p(\sigma)}^p \lVert f\rVert_{L^p(\sigma)}^p .
  \end{align*}
  Since the norm of the dyadic maximal function $\lVert M^{\sigma} \rVert_{L^p(\sigma) \to L^p(\sigma)} \le p'$
  and does not depend on the weight $\sigma$, the result follows.
  The estimate for the maximal function is classical, a proof in our case can be found in 
  \cite[Lemma 3.13]{MR4263006}.
\end{proof}
\begin{lemma}\label{lemma:aux2}
  Let $\mathscr{S}$ be a $\frac{1}{\tau}$-sparse family, $\tau \ge 1$.
  For $1<p<\infty$ let $\Psi \in \mathscr{B}_p$ be a Young function.
  Then for any $G \in \mathscr{S}$ the following estimate holds
  \begin{equation*}
    \sum_{\substack{Q \in \mathscr{S} \\ Q \subseteq G}} \langle w^{1/p} \rangle_{\Psi,Q}^p \lvert Q \rvert \lesssim w(G)
  \end{equation*}
  where the implicit constant depends only on $\tau$ and $\lVert M_{\Psi} \rVert_{L^p \to L^p}$. 
\end{lemma}
\begin{proof}
  By \cref{teo:Perez95}, since $\Psi \in \mathscr{B}_p$ the maximal function $M_{\Psi}$ is bounded on $L^p$.
  For $Q \subseteq G$, we have $\langle w^{1/p} \rangle_{\Psi,Q} = \langle w^{1/p} \1_G\rangle_{\Psi,Q}$.
  Then $\lvert Q \rvert \le \tau \lvert E_Q \rvert $ and
 \begin{align*}
   \sum_{\substack{Q \in \mathscr{S} \\ Q \subseteq G}} \langle w^{1/p} \rangle_{\Psi,Q}^p \lvert Q \rvert 
   & \le \tau \sum_{\substack{Q \in \mathscr{S} \\ Q \subseteq G}} \int_{E_Q} M_{\Psi}( w^{1/p} \1_G)^p  \\
   & \le \tau \int_{G} M_{\Psi}( w^{1/p} \1_G)^p  \\
   & \le \tau \lVert M_{\Psi} \rVert_{L^p \to L^p}^p \, \lVert w^{1/p} \rVert_{L^p(G)}^p .
 \end{align*}

\end{proof}

We are ready to prove \cref{prop:bump_for_sparse}.
We recall the two testing conditions in \eqref{eq:testing_condition_on_sparse}:
\begin{align*}
  \lVert \Lambda_{\mathscr{S}}(\sigma \1_Q) \rVert_{L^p(w)}^p & \le \testC \sigma(Q) \, , \quad \forall\, Q \in \mathscr{S} \\
  \lVert \Lambda_{\mathscr{S}}(w \1_Q) \rVert_{L^{p'}(\sigma)}^{p'} & \le \testC' w(Q)  \, , \quad \forall\, Q \in \mathscr{S} %
\end{align*}
By symmetry, it is enough to focus on one of the two.
\subsubsection*{1. Reduction to dyadic form}
By duality,
the left hand side of the two-weight estimate %
\begin{equation*}
  \lVert \Lambda(f\sigma) \rVert_{L^2(w)} \le C \lVert f\rVert_{L^2(\sigma)}
\end{equation*}
is the supremum over $g \in L^2(w)$ of $\lvert\langle \Lambda(f \sigma) , g w \rangle\rvert$.
Then it is enough to show that for non-negative functions $f$ and $g$ we have
\begin{align*}
  \lvert \langle \Lambda(f \sigma) , g w \rangle\rvert & = \sum_{Q \in \mathscr{S}} \langle f \sigma \rangle_Q \langle g w\rangle_Q \lvert Q \rvert \\
  & = \sum_{Q \in \mathscr{S}} \langle f \rangle_Q^{\sigma} \langle g \rangle_Q^{w} \langle \sigma \rangle_Q \langle w \rangle_Q \lvert Q \rvert 
  \lesssim [\sigma,w]_{\Phi,\Psi} \lVert f \rVert_{L^2(\sigma)} \lVert g \rVert_{L^2(w)}.
\end{align*}

\subsubsection*{2. Stopping families}
We assume that both $f,g$ are both non-negative and supported on the set $Q_0$.
Let $\mathcalboondox{D}(Q_0)$ be the family of dyadic cubes inside $Q_0$. 
We will select special cubes from $\mathcalboondox{D}(Q_0)$ using the ``parallel corona'' decomposition.
We denote the principal cubes for $(f,\sigma)$ and $(g,w)$ by $\mathscr{F}$
and  $\mathscr{G}$ respectively.
These are defined as stopping families for the weighted averages of $f$ and $g$:
\begin{align*}
  \mathcalboondox{A}_f^\star(Q) &= \{ S\in\mathcalboondox{D}(Q), S \text{ maximal} \, : \, \langle f \rangle_S^{\sigma} > 2 \langle f \rangle_{Q}^{\sigma} \} , \\
  \mathcalboondox{A}_g^\star(Q) &= \{ S\in\mathcalboondox{D}(Q), S \text{ maximal} \, : \, \langle g \rangle_S^{w} > 2 \langle g \rangle_{Q}^{w} \} .
\end{align*}
Then we define
\begin{equation*}
  \mathscr{F}_0 \coloneqq \{Q_0\}, \qquad \mathscr{F}_{n+1} \coloneqq \bigcup_{Q \in \mathscr{F}_n} \mathcalboondox{A}^\star_f(Q), \qquad \mathscr{F} \coloneqq \bigcup_{n \in \mathbb{N}} \mathscr{F}_n 
\end{equation*}
and in a similar way for $\mathscr{G}$.
The families $\mathscr{F}$ and $\mathscr{G}$ constructed in this way are sparse
with respect to the measures $\sigma \D{\nu}$ and $w \D{\nu}$, respectively.
We denote by
$\pi_{\mathscr{F}}(Q)$ the minimal cube in $\mathscr{F}$ containing
$Q$, and similarly for $\pi_{\mathscr{G}}(Q)$.  Given a pair of cubes
$(F,G) \in \mathscr{F}\times\mathscr{G}$, we consider the collection of cubes
such that their projection to $\mathscr{F}$ and $\mathscr{G}$ are $F$
and $G$ respectively. Such collection is
\begin{equation*}
   \{ Q \,\colon\, \pi(Q) =  (F,G) \}, \quad \text{where } \; \pi(Q) \coloneqq  (\pi_{\mathscr{F}}(Q),\pi_{\mathscr{G}}(Q)) .
\end{equation*}

Using the stopping families we can write
\begin{equation*}
  \sum_{Q \in \mathscr{S}} = \sum_{F \in \mathscr{F}} \sum_{G \in \mathscr{G}} \sum_{\substack{Q \in \mathscr{S} \\ \pi(Q) = (F,G)}}.
\end{equation*}
Since either $F \subseteq G$ or $F \supseteq G$,
by symmetry it is enough to study only one case.
We focus on the latter.
Notice that since $\pi_{\mathscr{G}}(Q) = G \subseteq F$, then $F$ is the minimal cube in $\mathscr{F}$ containing $G$,
namely $\pi_{\mathscr{F}}(G) = F$.
We have
\begin{align}
  \sum_{F \in \mathscr{F}} \sum_{\substack{G \in \mathscr{G} \\ \pi_{\mathscr{F}}(G) = F}} \sum_{\substack{Q \in \mathscr{S} \\ \pi(Q) = (F,G)}} & \langle f \rangle_Q^{\sigma} \langle g \rangle_Q^{w} \langle \sigma \rangle_Q \langle w \rangle_Q \lvert Q \rvert \nonumber \\
                                                                                                                                   & \le 4 \sum_{F \in \mathscr{F}} \langle f \rangle_F^{\sigma}  \sum_{\substack{G \in \mathscr{G} \\ \pi_{\mathscr{F}}(G) = F}} \langle g \rangle_G^{w}  \sum_{\substack{Q \in \mathscr{S} \\ \pi(Q) = (F,G)}} \langle \sigma \rangle_Q \langle w \rangle_Q \lvert Q \rvert.
  \label{eq:before_input_Orlicz_bumps}
 \end{align}

 \subsubsection*{3. Insert Orlicz bumps}
 We focus on the last summand in \eqref{eq:before_input_Orlicz_bumps}. We see that
 \begin{equation*}
   \langle \sigma \rangle_Q \langle w \rangle_Q =  \left(\frac{\langle \sigma \rangle_Q \langle w \rangle_Q}{\langle \sigma^{1/2} \rangle_{\Phi,Q} \langle w^{1/2} \rangle_{\Psi,Q}}\right) \langle \sigma^{1/2} \rangle_{\Phi,Q} \langle w^{1/2} \rangle_{\Psi,Q}.
 \end{equation*}
 The supremum over all dyadic cubes $Q$ of the quantity in brackets is $[\sigma,w]_{\Phi,\Psi}$.
 Then we have 
 \begin{equation*}
   \sum_{\substack{Q \in \mathscr{S} \\ \pi(Q) = (F,G)}} \langle \sigma \rangle_Q \langle w \rangle_Q \lvert Q \rvert
   \le [\sigma,w]_{\Phi,\Psi}  \sum_{\substack{Q \in \mathscr{S} \\ \pi(Q) = (F,G)}} \langle \sigma^{1/2} \rangle_{\Phi,Q} \langle w^{1/2} \rangle_{\Psi,Q} \lvert Q \rvert.
 \end{equation*}

 Using the Cauchy--Schwarz inequality and \cref{lemma:aux2} we estimate
 \begin{align*}
   \sum_{\substack{Q \in \mathscr{S} \\ \pi(Q) = (F,G)}} & \langle \sigma^{1/2} \rangle_{\Phi,Q} \langle w^{1/2} \rangle_{\Psi,Q} \lvert Q \rvert \\
   & \le \Big(\sum_{\substack{Q \in \mathscr{S} \\ \pi(Q) = (F,G)}} \langle \sigma^{1/2} \rangle_{\Phi,Q}^2 \lvert Q \rvert \Big)^{1/2} \Big( \sum_{\substack{Q \in \mathscr{S} \\ \pi(Q) = (F,G)}} \langle w^{1/2} \rangle_{\Psi,Q}^2 \lvert Q \rvert \Big)^{1/2} \\ %
   & \lesssim \Big(\sum_{\substack{Q \in \mathscr{S} \\ \pi(Q) = (F,G)}} \langle \sigma^{1/2} \rangle_{\Phi,Q}^2 \lvert Q \rvert \Big)^{1/2} w(G)^{1/2} .
 \end{align*}

 Putting all the estimates together, and using the Cauchy--Schwarz inequality in $\ell^2$ in the third and and fifth inequality
 and \cref{lemma:aux2} in the second and the fourth,
 we obtain
 \begin{align*}
   \sum_{F \in \mathscr{F}} & \langle f \rangle^\sigma_F  \sum_{\substack{G \in \mathscr{G} \\ \pi_{\mathscr{F}}(G) = F}} \langle g \rangle^w_G \sum_{\substack{Q \in \mathscr{S} \\ \pi(Q) = (F,G)}} \langle \sigma \rangle_Q \langle w \rangle_Q \lvert Q \rvert \\
   & \le  [\sigma,w]_{\Phi,\Psi}  \sum_{F \in \mathscr{F}} \langle f \rangle^\sigma_F  \sum_{\substack{G \in \mathscr{G} \\ \pi_{\mathscr{F}}(G) = F}} \langle g \rangle^w_G  \sum_{\substack{Q \in \mathscr{S} \\ \pi(Q) = (F,G)}} \langle \sigma^{1/2} \rangle_{\Phi,Q} \langle w^{1/2} \rangle_{\Psi,Q} \lvert Q \rvert \\
  & \lesssim  [\sigma,w]_{\Phi,\Psi}  \sum_{F \in \mathscr{F}} \langle f \rangle^\sigma_F  \sum_{\substack{G \in \mathscr{G} \\ \pi_{\mathscr{F}}(G) = F}} \langle g \rangle^w_G  \Big(\sum_{\substack{Q \in \mathscr{S} \\ \pi(Q) = (F,G)}} \langle \sigma^{1/2} \rangle_{\Phi,Q}^2  \lvert Q \rvert\Big)^{1/2} w(G)^{1/2} \\
   & \le  [\sigma,w]_{\Phi,\Psi}  \sum_{F \in \mathscr{F}} \langle f \rangle^\sigma_F  \Big(\sum_{\substack{G \in \mathscr{G} \\ \pi_{\mathscr{F}}(G) = F}} (\langle g \rangle^w_G)^2 w(G) \Big)^{1/2}
   \Big(\sum_{\substack{G \in \mathscr{G} \\ \pi_{\mathscr{F}}(G) = F}} \sum_{\substack{Q \in \mathscr{S} \\ \pi(Q) = (F,G)}} \langle \sigma^{1/2} \rangle_{\Phi,Q}^2  \lvert Q \rvert\Big)^{1/2}  \\
    & \lesssim  [\sigma,w]_{\Phi,\Psi}  \sum_{F \in \mathscr{F}} \langle f \rangle^\sigma_F  \Big(\sum_{\substack{G \in \mathscr{G} \\ \pi_{\mathscr{F}}(G) = F}} (\langle g \rangle^w_G)^2 w(G) \Big)^{1/2} \sigma(F)^{1/2}\\
    & \le  [\sigma,w]_{\Phi,\Psi}  \Big(\sum_{F \in \mathscr{F}} (\langle f \rangle^\sigma_F)^2 \sigma(F) \Big)^{1/2} \Big(\sum_{F \in \mathscr{F}} \sum_{\substack{G \in \mathscr{G} \\ \pi_{\mathscr{F}}(G) = F}} (\langle g \rangle^w_G)^2 w(G) \Big)^{1/2} \\
   & \lesssim  [\sigma,w]_{\Phi,\Psi} \lVert f \rVert_{L^2(\sigma)} \lVert g\rVert_{L^2(w)} 
 \end{align*}
 where the last inequality follows from  \cref{lemma:aux1},
 concluding the proof. \qed

\section*{Acknowledgements}
This work is part of the author's PhD thesis
and it has been initiated in Ghent during the (unplanned) staying of the author.
I would like to thank Angelos Nersesian, Michele Mastropietro and the staff at the University of Ghent %
for their warm hospitality.
I would also like to thank Maria Carmen Reguera
for stimulating (virtual) discussions during those months,
and the anonymous referees
for providing important suggestions that improve the article.

\nocite{PR13,APR17,RTW17, CMP07, FangWang2015, HytKairema}
\printbibliography

\end{document}